\documentclass[10pt]{amsart}
\usepackage{amscd}
\usepackage{amssymb}
\usepackage[all]{xy}
\usepackage{mathpazo} 
\usepackage[T1]{fontenc}
\usepackage{hyperref}
\date{25 April 2014 bis}
\title[Complete Modules]{A Separated Cohomologically Complete Module is 
Complete}
\author{Amnon Yekutieli}
\address{Department of  Mathematics,
Ben Gurion University, Be'er Sheva 84105, Israel}
\email{amyekut@math.bgu.ac.il}

\thanks{{\em Mathematics Subject Classification} 2010.
Primary: 13B35; Secondary: 18G10, 13J10.}
\keywords{adic completion, derived completion}  
\newtheorem{thm}[equation]{Theorem}

\newtheorem{lem}[equation]{Lemma}
\theoremstyle{definition}

\newtheorem{rem}[equation]{Remark}
\newtheorem{exa}[equation]{Example}

\newcommand{\xar}{\xrightarrow}
\newcommand{\opn}{\operatorname}
\newcommand{\cat}[1]{\operatorname{\mathsf{#1}}}

\newcommand{\rmitem}[1]{\item[\text{\textup{(#1)}}]}
\newcommand{\mfrak}[1]{\mathfrak{#1}}

\newcommand{\mrm}[1]{\mathrm{#1}}

\newcommand{\La}{\Lambda}

\newcommand{\de}{\delta}

\renewcommand{\b}{\mfrak{b}}
\renewcommand{\a}{\mfrak{a}}

\newcommand{\K}{\mathbb{K}}

\newcommand{\Z}{\mathbb{Z}}
\newcommand{\N}{\mathbb{N}}

\newcommand{\DL}{\mrm{L}}

\newcommand{\tup}[1]{\textup{#1}}
\newcommand{\bsym}[1]{\boldsymbol{#1}}

\newcommand{\ot}{\otimes}

\newcommand{\what}[1]{\widehat{#1}}

\renewcommand{\d}{\mathrm{d}}

\newcommand{\lb}{\linebreak}

\begin{document}

\begin{abstract}
We prove two theorems on cohomologically complete complexes. \lb These theorems 
are inspired by, and yield an alternative proof of, a recent theorem of P. 
Schenzel on complete modules. 
\end{abstract}

\maketitle

Throughout $A$ denotes a nonzero commutative ring. (We do not assume $A$ is 
noetherian.) The category of $A$-modules is denoted by $\cat{Mod} A$, 
its category of unbounded complexes is $\cat{C}(\cat{Mod} A)$,
and the derived category is $\cat{D}(\cat{Mod} A)$. 

Let $\a$ be a finitely generated ideal in $A$. 
Recall that the $\a$-adic completion of an $A$-module $M$ is 
\[ \La_{\a}(M) := \lim_{\leftarrow k} \bigl( (A / \a^k) \ot_A M \bigr) . \] 
The module $M$ is called {\em $\a$-adically complete} (resp.\ {\em 
$\a$-adically separated}) if the canonical homomorphism 
$\tau_{M} : M \to \La_{\a}(M)$ is bijective (resp.\ injective). 
It is known that the completion $\La_{\a}(M)$ is $\a$-adically complete
(see \cite[Corollary 3.6]{Ye}).

The additive functor 
\[ \La_{\a} : \cat{Mod} A \to \cat{Mod} A \]
has a left derived functor
\[ \DL \La_{\a} : \cat{D}(\cat{Mod} A) \to \cat{D}(\cat{Mod} A) . \]
A complex $M \in \cat{D}(\cat{Mod} A)$ is called {\em cohomologically 
$\a$-adically complete} if the canonical morphism 
$\tau_{M}^{\DL} : M \to \DL \La_{\a} (M)$ in $\cat{D}(\cat{Mod} A)$ is an 
isomorphism. We say that an $A$-module $M$ is cohomologically complete if it is 
so as a complex,  using the standard embedding 
$\cat{Mod} A \to \cat{D}(\cat{Mod} A)$.
See \cite{PSY1} for more details. 

The next example (taken from \cite{Ye}) shows an anomalous
module $M$: it is cohomologically $\a$-adically complete, but not 
$\a$-adically complete.

First we have to recall some concepts from \cite{Ye}.
Assume $A$ is $\a$-adically complete, and let $Z$ be some set. 
The module of finitely supported functions
$f : Z \to A$, denoted by   $\opn{F}_{\mrm{fin}}(Z, A)$, is free with basis 
the collection $\{ \de_z \}_{z \in Z}$ of delta functions. Its $\a$-adic 
completion is canonically isomorphic to the module 
$\opn{F}_{\mrm{dec}}(Z, A)$ of {\em $\a$-adically decaying functions}; see 
\cite[Definition 2.1 and Corollary 2.9]{Ye}.
An $A$-module $P$ is called {\em $\a$-adically free} if it is isomorphic to 
$\opn{F}_{\mrm{dec}}(Z, A)$ for some set $Z$.

\begin{exa} \label{exa:1}
Let $A := \K[[t]]$, the power series ring in a variable $t$ over a field
$\K$, and $\a := (t)$. 
Define the $\a$-adically free $A$-modules 
$P^0 = P^{-1} := \opn{F}_{\mrm{dec}}(\N, A)$.
Let $\d : P^{-1} \to P^0$ be the homomorphism 
$\d(\de_i) := t^i \de_i$. 
Define $P := ( P^{-1} \xar{ \ \d \ } P^0 )$, a complex concentrated in degrees 
$-1, 0$; and let $M := \opn{H}^0(P) \in \cat{Mod} A$. 

Consider the canonical surjection $\pi^0 : P^0 \to M$, and  
$m := \pi^0(\sum_{i \in \N} t^i \de_i) \in M$.
The element $m$ is nonzero, but it belongs to 
$\cap_{i \in \N} \a^i M$. 
Therefore $M$ is not an $\a$-adically separated module, and hence it is not an 
$\a$-adically complete module. 
(This was already noticed in \cite[Example 2.5]{Si2}.)

On the other hand, the canonical homomorphism $\pi : P \to M$ is a 
quasi-isomorph\-ism, so according to \cite[Theorem 1.15]{PSY2} the module $M$ 
is cohomologically $\a$-adically complete.
\end{exa}

\begin{rem}
Here are a few words regarding the history and background. The total left 
derived functor $\DL \La_{\a}$ was first studied in \cite{AJL}, following 
earlier work on the left derived functors 
$\DL_i \La_{\a} = \opn{H}^{-i}(\DL \La_{\a})$, 
mostly in \cite{GM}. See also \cite{Ma} and \cite{Si1}.

Many important properties of $\DL \La_{\a}$ can be found in the paper 
\cite{AJL}. However, we prefer to quote \cite{PSY1, PSY2}, where the relevant 
theory was developed further. 

The definition of cohomologically complete complexes above was introduced in 
\cite{PSY1}. The name actually originates in \cite{KS}, but the definition 
there is different (yet equivalent, as proved in \cite[Theorem 1.4]{PSY1}).
\end{rem}

Recently we came across a remarkable new result of P. Schenzel
(\cite[Theorem 1.1]{Sc}, which is the noetherian case of Theorem \ref{thm:4} 
below). While trying to understand this result, we 
discovered the next theorem, which explains the anomaly of the 
module $M$ in Example \ref{exa:1}. 

Let us denote by $\cat{D}^{\mrm{b}}(\cat{Mod} A)$ the full subcategory of 
$\cat{D}(\cat{Mod} A)$ consisting of complexes with bounded cohomologies.

\begin{thm} \label{thm:2}
Let $A$ be a noetherian commutative ring, let $\a$ be an ideal in $A$, and let 
$M \in \cat{D}^{\mrm{b}}(\cat{Mod} A)$. The following conditions are 
equivalent\tup{:}
\begin{enumerate}

\rmitem{i} The $A$-modules $\opn{H}^j(M)$ are $\a$-adically complete for all
$j \in \Z$. 

\rmitem{ii} The complex $M$ is $\a$-adically cohomologically complete, and the 
$A$-modules \lb $\opn{H}^j(M)$ are $\a$-adically separated for all $j \in \Z$.

\end{enumerate}
\end{thm}

\begin{proof}
(i) $\Rightarrow$ (ii): This is immediate from 
\cite[Theorem 1.21]{PSY2}.

\medskip \noindent
(ii) $\Rightarrow$ (i): Here we have to do some work. 
Recall that for a graded $A$-module $N = \bigoplus_{i \in \Z} N^i$ its 
amplitude is 
\[ \opn{amp}(N) := \opn{sup}(N) - \opn{inf}(N) \in \Z \cup \{ \pm \infty \} . \]
The amplitude satisfies 
$\opn{amp}(N) = - \infty$ iff $N = 0$, and $\opn{amp}(N) < \infty$ iff $N$ is 
bounded. Cf.\ \cite[formulas (2.3)-(2.5)]{PSY1}. We proceed by induction on \lb
$\opn{amp}(\opn{H}(M))$. 

If $M = 0$ there is nothing to prove. So let us assume that 
$0 \leq \opn{amp}(\opn{H}(M)) < \infty$. Let $j := \opn{sup}(\opn{H}(M))$. 
According to \cite[Theorem 1.15]{PSY2}, there a quasi-isomorphism 
$P \to M$, where $P$ is a complex of $\a$-adically free modules,
and $\opn{sup}(P) = j$. 
Let $\pi : P^j \to \opn{H}^j(P)$ be the canonical surjection.
Since $\tau : \bsym{1} \to \La_{\a}$ is a natural transformation, 
there is a commutative diagram 
\[ \UseTips \xymatrix @C=8ex @R=5ex {
P^j
\ar[r]^(0.5){\pi}
\ar[d]_{\tau_{P^j}}
&
\opn{H}^j(P)
\ar[d]^{\tau_{\opn{H}^j(P)}}
\\
\La_{\a}(P^j)
\ar[r]^(0.4){\La_{\a}(\pi)}
&
\La_{\a}(\opn{H}^j(P))
\ \  .
} \]
We know that the functor $\La_{\a}$ 
preserves surjections, so $\La_{\a}(\pi)$ is also surjective. Since $P^j$ is 
complete, the homomorphism $\tau_{P^j}$ is bijective. 
Therefore $\tau_{\opn{H}^j(P)}$ is surjective. 
On the other hand, by assumption the module 
$\opn{H}^j(P) \cong \opn{H}^j(M)$ is separated, so the homomorphism 
$\tau_{\opn{H}^j(P)}$ is injective. Thus $\tau_{\opn{H}^j(P)}$ is 
bijective, and we conclude that $\opn{H}^j(M)$ is 
$\a$-adically complete.

Using smart truncation of $M$ at $j$ there is a distinguished triangle 
\[ M' \to M \to \opn{H}^j(M)[-j] \to M'[1] \]
in $\cat{D}(\cat{Mod} A)$, such that 
$\opn{H}^{j'}(M') \cong \opn{H}^{j'}(M)$ for all $j' < j$, and 
$\opn{H}^{j'}(M') = 0$ for all $j' \geq j$. By the implication 
(i) $\Rightarrow$ (ii) we know that the module $\opn{H}^j(M)$ is 
cohomologically complete. Since the category 
$\cat{D}(\cat{Mod} A)_{\a \tup{-com}}$ is triangulated, it follows that $M'$ is
also cohomologically complete. But 
$\opn{amp}(\opn{H}(M')) < \opn{amp}(\opn{H}(M))$, so induction tells us that 
the modules $\opn{H}^{j'}(M')$ are all complete. 
\end{proof}

\begin{rem}
We do not know whether Theorem \ref{thm:2} holds without assuming that 
$\opn{H}(M)$ is bounded. Perhaps ideas in \cite{Si1} can shed some light on 
this 
question. 
\end{rem}

Here is our second new result on cohomologically complete complexes. 

\begin{thm} \label{thm:3}
Let $A$ be a noetherian commutative ring, let $\a_1, \ldots, \a_n$ 
be ideals in $A$,  let $\a := \a_1 + \cdots + \a_n$, and
let $M \in \cat{D}(\cat{Mod} A)$. The following conditions are 
equivalent\tup{:}
\begin{enumerate}
\rmitem{i} $M$ is $\a$-adically cohomologically complete. 
\rmitem{ii} $M$ is $\a_i$-adically cohomologically complete for all
$i = 1, \ldots, n$.
\end{enumerate}
\end{thm}

\begin{proof}
For every $i$ choose a finite sequence $\bsym{a}_i$ that generates the ideal 
$\a_i$. Let $\bsym{a} := \lb (\bsym{a}_1, \ldots, \bsym{a}_n)$, the 
concatenated sequence, and let $\bsym{b}_i$ be the sequence gotten from 
$\bsym{a}$ be deleting $\bsym{a}_i$.
Define 
$T := \opn{Tel}(A; \bsym{a})$, 
$T_i := \opn{Tel}(A; \bsym{a}_i)$ and
$S_i := \opn{Tel}(A; \bsym{b}_i)$, the telescope complexes
from \cite[Definition 5.1]{PSY1}. Note that
\[ T \cong T_{1} \ot_A \cdots \ot_A T_{n} \cong T_i \ot_A S_i \]
in $\cat{C}(\cat{Mod} A)$. 

There is a canonical homomorphism $u : T \to A$
in $\cat{C}(\cat{Mod} A)$, see 
\cite[formula 5.6]{PSY1}. According to \cite[Corollary 5.25]{PSY1},
$M$ is $\a$-adically cohomologically complete iff the homomorphism 
\[ \opn{Hom}(u, \bsym{1}_M) : M \to  \opn{Hom}_A(T, M) \]
in $\cat{C}(\cat{Mod} A)$ is a quasi-isomorphism. 
Likewise there are homomorphisms $u_i : T_i \to A$, and 
$M$ is $\a_i$-adically cohomologically complete iff the homomorphism 
\[ \opn{Hom}(u_i, \bsym{1}_M) : M \to  \opn{Hom}_A(T_i, M) \]
is a quasi-isomorphism. 

Let us prove the implication (i) $\Rightarrow$ (ii).
For every index $i$ there is a diagram 
\[ \UseTips \xymatrix @C=16ex @R=8ex {
M
\ar[r]^{ \opn{Hom}(u, \bsym{1}_M) }
\ar[d]_{ \opn{Hom}(u_i, \bsym{1}_M) }
&
\opn{Hom}_A(T, M)
\ar[d]^{ \opn{Hom}(\bsym{1}_T \ot u_i, \bsym{1}_M) }
\\
\opn{Hom}_A(T_i, M)
\ar[r]^(0.45){ \opn{Hom}(u \ot_A \bsym{1}_{T_i}, \bsym{1}_M) }
&
\opn{Hom}_A(T \ot_A T_i, M)
\ \  .
} \]
in $\cat{C}(\cat{Mod} A)$, which is commutative up to sign. 
By assumption the homomorphism 
$\opn{Hom}(u, \bsym{1}_M)$ is a quasi-isomorphism. 
By adjunction there is an isomorphism 
\[ \opn{Hom}_A(T \ot_A T_i, M) \cong \opn{Hom}_A(T_i, \opn{Hom}_A(T, M)) \]
in $\cat{C}(\cat{Mod} A)$, and it sends 
\[ \opn{Hom}(u \ot_A \bsym{1}_{T_i}, \bsym{1}_M) \mapsto
\pm \opn{Hom}(\bsym{1}_{T_i}, \opn{Hom}_A(u, \bsym{1}_M)) . \]
Because $T_i$ is K-projective, we see that 
$\opn{Hom}(u \ot_A \bsym{1}_{T_i}, \bsym{1}_M)$ is a quasi-iso\-morphism. 

By \cite[Lemma 7.9]{PSY1} the homomorphism 
$u_i : T_i \ot_A T_i \to T_i$ is a homotopy equivalence.
This implies that 
\[ \bsym{1}_T \ot u_i  = \pm  \bsym{1}_{T_i} \ot \bsym{1}_{S_i} \ot u_i : 
S_i \ot_A T_i \ot_A T_i \to S_i \ot_A T_i  \]
is a  homotopy equivalence, and therefore 
$\opn{Hom}(\bsym{1}_T \ot u_i, \bsym{1}_M)$
is a quasi-iso\-morph\-ism. 
The conclusion is that $\opn{Hom}(u_i, \bsym{1}_M)$ is a quasi-iso\-morph\-ism. 

Finally we prove that (ii) $\Rightarrow$ (i).
The assumption is that $M \to \opn{Hom}_A(T_i, M)$ are quasi-isomorphisms.
Applying $\opn{Hom}_A(T_2, -)$ to the quasi-isomorph\-ism 
$M \to \opn{Hom}_A(T_1, M)$, we get a quasi-isomorphism 
\[ \opn{Hom}_A(T_2, M) \to \opn{Hom}_A(T_2, \opn{Hom}_A(T_1, M)) . \]
Continuing this way we end up with a sequence of quasi-isomorphisms
\[ M \to \opn{Hom}_A(T_n, M) \to 
\cdots \to \opn{Hom}_A(T_{n}, \cdots, \opn{Hom}_A(T_1, M) \cdots ) .  \]
Using adjunction we get an isomorphism 
\[ \opn{Hom}_A(T_{n}, \cdots, \opn{Hom}_A(T_1, M) \cdots ) \cong 
\opn{Hom}_A(T_{1} \ot_A \cdots \ot_A T_{n}, M)   \]
in $\cat{C}(\cat{Mod} A)$. 
Up to sign, the resulting quasi-isomorphism 
$M \to \opn{Hom}_A(T, M)$ is \lb $\opn{Hom}(u , \bsym{1}_M)$. 
\end{proof}

\begin{rem}
In Theorems \ref{thm:2} and \ref{thm:3} we can remove the assumption that the 
ring $A$ is noetherian, and replace it by the weaker assumption that the ideals 
$\a, \a_1, \ldots, \a_n$ are {\em weakly proregular}; see \cite{PSY1}. 
This observation was communicated to us by L. Shaul.
\end{rem}

\begin{lem} \label{lem:1}
Assume $A$ is noetherian, and $\a$ is a principal ideal, generated by 
an element $a$. Let $M$ be an $A$-module. 
\begin{enumerate}
\item The following conditions are equivalent\tup{:}
\begin{enumerate}
\rmitem{i} $M$ is $\a$-adically cohomologically complete.
\rmitem{ii}  $\opn{Ext}^j_A(A_{a}, M) = 0$ for $j = 0, 1$.
\end{enumerate}

\item If $M$ is $\a$-adically separated, then $\opn{Ext}^0_A(A_{a}, M) = 0$.
\end{enumerate}
\end{lem}

In the lemma, $A_{a}$ is the localization of $A$ w.r.t.\ the element $a$.

\begin{proof}
(1) In the principal case the DG algebra $\opn{C}(A; \bsym{a})$ appearing in 
\cite[Section 8]{PSY1} is just the localized ring $A_{a}$. Thus by 
\cite[Theorem 8.8]{PSY1}, $M$ is $\a$-adically cohomologically complete iff 
$\opn{RHom}_{A}(A_{a}, M) = 0$. 

The telescope complex $\opn{Tel}(A; a)$ from \cite[Definition 
5.1]{PSY1} is a complex of free $A$-modules 
\[ \opn{Tel}(A; a) = \bigl( \cdots \to 0 \to \opn{F}_{\mrm{fin}} (\N, A) \to 
\opn{F}_{\mrm{fin}} (\N, A) \to  0 \to  \cdots \bigr) \]
concentrated in degrees $0, 1$.
The infinite dual Koszul complex
\[ \opn{K}_{\infty}^{\vee}(A; a) = \bigl( \cdots \to 0 \to 
 A \to A_a \to  0 \to  \cdots \bigr) \]
is a complex of flat $A$-modules, also concentrated in degrees $0, 1$.
According to \cite[Lemma 5.7]{PSY1} there is a quasi-isomorphism 
\[ w_a : \opn{Tel}(A; a) \to \opn{K}_{\infty}^{\vee}(A; a) . \]
Looking at the definitions of $\opn{Tel}(A; a)$ and $w_a$, we see that passing 
to the subcomplex
\[ \opn{Tel}_+(A; a) := \bigl(  \cdots \to 0 \to 
\opn{F}_{\mrm{fin}} ([1, \infty], A) \to 
\opn{F}_{\mrm{fin}} (\N, A) \to  0 \to  \cdots \bigr)  \]
of $\opn{Tel}(A; a)$, i.e.\ omitting the module $A \de_0$ in degree $0$, 
we get an induced a quasi-isomorphism 
\[ w_a : \opn{Tel}_+(A; a) \to A_a[-1] .  \]
Therefore  
\[ \opn{RHom}_{A}(A_{a}, M) \cong 
\opn{Hom}_{A} \bigl( \opn{Tel}_+(A; a)[1], M \bigr) \]
has cohomology only in degrees $0, 1$. 

\medskip \noindent
(2) Assume $M$ is $\a$-adically separated, and let 
$\what{M} := \La_{\a}(M)$. Since $\tau_{M} : M \to \what{M}$ is injective, 
we have an injection 
$\opn{Hom}_A(A_{a}, M) \to \opn{Hom}_A(A_{a}, \what{M})$. 
But $\what{M}$ is $\a$-adically complete, so by Theorem \ref{thm:2} it is 
$\a$-adically cohomologically complete, and thus by part (1) we know that 
$\opn{Hom}_A(A_{a}, \what{M}) = 0$.
\end{proof}

\begin{lem} \label{lem:5}
Suppose $f : A \to B$ is a ring homomorphism, $a \in A$, $b := f(a) \in B$, and 
$M \in \cat{Mod} B$. For every $i$ there is a canonical $B$-module isomorphism 
\[ \opn{Ext}^i_B(B_{b}, M) \cong \opn{Ext}^i_A(A_{a}, M) . \]
\end{lem}

\begin{proof}
There is an isomorphism of complexes of $B$-modules
\[ B \ot_A \opn{Tel}_+(A; a) \cong \opn{Tel}_+(B; b) . \]
This induces an isomorphism 
\[ \begin{aligned}
& \opn{Ext}^i_B(B_{b}, M) \cong
\opn{H}^i \bigl( \opn{Hom}_{B} \bigl( \opn{Tel}_+(B; b)[1], M \bigr) \bigr)
\\
& \qquad \cong  
\opn{H}^i \bigl( \opn{Hom}_{A} \bigl( \opn{Tel}_+(A; a)[1], M \bigr) \bigr)
\cong \opn{Ext}^i_A(A_{a}, M) .
\end{aligned} \]
\end{proof}

\begin{lem} \label{lem:2}
Let $\b \subset \a$. If an $A$-module $M$ is $\a$-adically separated, then it 
is also $\b$-adically separated.
\end{lem}

\begin{proof}
Since $\b^i M \subset \a^i M$ for all $i$, we have 
$\bigcap_i \b^i M \subset \bigcap_i \a^i M = 0$. 
\end{proof}

Schenzel's \cite[Theorem 1.1]{Sc} is the next theorem, when $A$ is noetherian.
This condition turns out to be unnecessary.  

\begin{thm} \label{thm:4}
Let $A$ be a commutative ring, let $\a$ be a finitely generated ideal in $A$, 
and let $(a_1, \ldots, a_n)$ be a sequence of elements that generates $\a$. 
The following conditions are equivalent for any $A$-module $M$~\tup{:}
\begin{enumerate}
\rmitem{i} $M$ is $\a$-adically complete. 

\rmitem{ii} $M$ is $\a$-adically separated, and 
$\opn{Ext}^1_A(A_{a_i}, M) = 0$ for all $i = 1, \ldots, n$.
\end{enumerate}
\end{thm}

\begin{proof}
Step 1. Assume $A$ is noetherian. For any $i$ let $\a_i$ be the ideal generated 
by the elements $a_i$. Consider these further conditions:
\begin{enumerate}
\item[(n1)] $M$ is $\a$-adically separated and $\a$-adically cohomologically 
complete.

\item[(n2)] $M$ is $\a$-adically separated, and $\a_i$-adically 
cohomologically complete for all  $i = 1, \ldots, n$.

\item[(n3)] $M$ is $\a$-adically separated, and $\a_i$-adically 
complete for all  $i = 1, \ldots, n$.
\end{enumerate}
By Theorem \ref{thm:2} we have (i) $\Leftrightarrow$ (n1).
By Theorem \ref{thm:3} we have (n1) $\Leftrightarrow$ (n2).
Combining Lemma \ref{lem:2} and Theorem \ref{thm:2} we deduce the equivalence 
(n2) $\Leftrightarrow$ (n3).
Finally, the equivalence (n3) $\Leftrightarrow$ (ii) comes from the combination 
of Lemmas \ref{lem:2} and \ref{lem:1}.

\medskip \noindent
Step 2. Now $A$ is arbitrary. 
Consider the polynomial ring 
$\Z[\bsym{t}] := \Z[t_1, \ldots, t_n]$, the ideal 
$\mfrak{t} := (t_1, \ldots, t_n)$, and the ring homomorphism 
$f : \Z[\bsym{t}] \to A$, $f(t_i) := a_i$. 
Since $\a^k M = \mfrak{t}^k M$ for every $k \in \N$, we see that $M$ is 
$\a$-adically complete (resp.\ separated) iff it is $\mfrak{t}$-adically 
complete (resp.\ separated). On the other hand, by Lemma \ref{lem:5} we know 
that $\opn{Ext}^1_A(A_{a_i}, M) = 0$ iff 
$\opn{Ext}^1_{\Z[\bsym{t}]} (\Z[\bsym{t}]_{t_i}, M) = 0$. 
Since the ring $\Z[\bsym{t}]$ is noetherian, we are done by step 1.
\end{proof}

\medskip \noindent
{\bf Acknowledgments}. 
I wish to thank Liran Shaul for helpful discussions.
Thanks also to the anonymous referee for a careful reading of the paper and for 
several suggestions.

%\cleardoublepage

\end{document}